\documentclass[10pt,a4paper]{article}
\usepackage[utf8]{inputenc}
\usepackage{amsmath}
\usepackage{amsfonts}
\usepackage{amssymb}
\usepackage{graphicx}
\usepackage{titlesec}
\usepackage[a4paper,left=2.75cm,right=2.75cm,top=2.75cm,bottom=2.75cm]{geometry}

\usepackage{lipsum}
\usepackage{ntheorem}
\providecommand{\keywords}[1]{\textbf{{Keywords:}} #1}
\usepackage{caption}

\usepackage{multirow}

\newcommand{\qedwhite}{\hfill \ensuremath{\Box}}

\usepackage{footmisc}

\titlelabel{\thetitle.~}
\numberwithin{equation}{section}
\newtheorem{theorem}{Theorem}[section]{\bfseries}{\itshape}
\newtheorem*{proof}{Proof}
\newtheorem{example}{Example}[section]{\bfseries}{\itshape}
\newcommand*\diff{\mathop{}\!\mathrm{d}}

\newcommand{\Int}{\int\limits}

\author{Mansur I. Ismailov{\footnotesize *$^,$**}, Sait Erkovan{\footnotesize *} \\
{\footnotesize * Gebze Technical University, Department of Mathematics, 41400, Gebze/Kocaeli, Turkey,} \\
{\footnotesize ** Institute of Mathematics and Mechanics, National Academy of Sciences of Azerbaijan, AZ1141, Baku, Azerbaijan,}\\
\footnotesize  mismailov@gtu.edu.tr, serkovan@gtu.edu.tr}
\title{\textbf{Inverse Problem of Finding the Coefficient of the Lowest Term in Two-dimensional Heat Equation with Ionkin-type Boundary Condition}}
\date{}
\begin{document}

\maketitle
\begin{abstract}
We consider an inverse problem of determining the time-dependent lowest order coefficient of two-dimensional (2D) heat equation with Ionkin boundary and total energy integral overdetermination condition. The well-posedness of the problem is obtained by generalized Fourier method combined by the Banach fixed poind theorem. For obtaining a numerical solution of the inverse problem, we propose the discretization method from a new combination. On the one hand, it is known the traditional method of uniform finite difference combined with numerical integration on a uniform grid (trapezoidal and Simpson's), on the other hand, we give the method of non-uniform finite difference is combined by a numerical integration on a non-uniform grid (with Gauss-Lobatto nodes). Numerical examples illustrate how to implement the method.

~~\\

~~
\end{abstract}

\section{Introduction and Problem Formulation}
Many researchers have been involved in the inverse coefficient problem for heat equation since this topic is of importance in some engineering texts and many industrial applications. The 1D problems have been well investigated. However, there are so many works which are needed to investigate for the multidimentional problems (see the books \cite{a2.,a1.,a3.} and references therein), although many researchers have reported its difficulties.
   
   Among these inverse problems, much attention is given to the determination of the lowest order coefficient in heat equation, in particular, when this coefficient depends solely on time. 
Various methods for finding the lowest order coefficient in a more general multidimensional parabolic equation have been addressed in numerous works, see \cite{BK15,CL882,CL88,CLX94,Kam15} for time dependent coefficient, \cite{Cho94,Kam13,Kos15,Koz03,Koz06,Koz02,PS87} for space dependent coefficient, \cite{Koz04,PK93,Pya04} for both time and space dependent coefficient. The boundary conditions are most frequently classical (Dirichlet, Neumann and Robin) and additional condition is most frequently specified as the solution at an interior point or an integral mean over the entire domain. 
    
    In this paper, we consider the problem of determining the lowest coefficient that depends on time only,  for a two-dimensional parabolic equation with Ionkin type nonlocal boundary condition and the total energy measurement. This type of nonlocal boundary conditions for partial differential equations were considered by numerous authors starting from the classical work \cite{BS69}.  Nonlocal problems have the specific feature that the corresponding spatial differential operator is nonself-adjoint and hence the system of eigenfunctions is not complete and must be supplemented by associated functions. The existence and uniqueness of the solution of such an inverse problem and well-posedness of this problem are examined by using the method of series expansion in terms of eigenfunctions and associated functions of mentioned spatial differential operator.
    
Numerical methods for solving the problem of the identification of the lowest coefficient of 1D parabolic equations are considered in many works. The works \cite{[1],[4],[2],[3]} dealing with the numerical solution of inverse problems for parabolic equations for an application point of view. We highlight separately studies dealing with numerical solving inverse problems of finding time-dependent lowest coefficient for multidimensional parabolic equations \cite{DS05,Deh00,3.} in view of the practical use. 
Discretization in our problem is performed using uniform and non-uniform finite difference method combined with suitable nodes for numerical integration on uniform and non-uniform grids. 

We study the following problem for two-dimensional heat equation with Ionkin-type boundary condition:\\
Let $D_{xy}=\{(x,y):0<x,y<1\}$. In the domain $D=D_{xy}\times(0<t\leq T)$, we consider
\begin{equation}\label{1}
\frac{\partial u}{\partial t}=\frac{\partial^2 u}{\partial x^2}+\frac{\partial^2 u}{\partial y^2}-p(t)u(x,y,t)+f(x,y,t)
\end{equation}
with the initial condition
\begin{equation}\label{2}
u(x,y,0)=\varphi(x,y),\quad 0\leq x,y\leq 1,
\end{equation}
nonlocal boundary conditions
\begin{equation}\label{3}
\begin{split}
u(0,y,t)=u(1,y,t),\quad u_x(1,y,t)=0,\quad 0\leq y\leq 1, \quad 0\leq t\leq T, \\
u(x,0,t)=u(x,1,t)=0, \quad 0\leq x\leq 1,\quad 0\leq t\leq T, \qquad\qquad\quad
\end{split}
\end{equation}
and overdetermination condition
\begin{equation}\label{4}
\Int_0^1\Int_0^1 u(x,y,t)\diff x\diff y=E(t).
\end{equation}
The problem of finding a pair $\{p(t),u(x,y,t)\}\in C[0,T]\times C^{2,2,1}(\overline{D})$ in \eqref{1}-\eqref{4} will be called an inverse problem.

The paper is organized as follows. In Section 2, we recall some necessary results on basisness of root functions concerning to two-dimensional spectral problem with Ionkin-type boundary condition, on finite difference discretization on uniform and non-uniform grids and on some product rules of numerical integration. The well-posedness of inverse problem \eqref{1}-\eqref{4} for small $T$ is showed by using generalized Fourier method combined with Banach fixed point theorem in Section 3. The numerical methods for solving the inverse problem \eqref{1}-\eqref{4} by applying uniform and non-uniform finite difference with suitable numerical integration rules are given in Section 4. We also present several numerical examples intended to illustrate the behaviour of the proposed methods. The tests were performed using MATLAB are discussed in Section 4. Finally, the concluding remark with the comparasion of two different numerical methods on uniform and non-uniform grids are presented in Section 5.

\section{Preliminaries}
In this section, we recall some necessary results on basisness of root functions concerning to two-dimensional spectral problem with Ionkin-type boundary conditions (see \cite{IM00}), on finite difference discretization on uniform and non-uniform grids (see \cite{BS05,GIM01}) and on some product rules of numerical integration which are used for solving the problem numerically.

\subsection{Spectral problem}
Consider the following spectral problem:
\begin{equation} \label{7}
\frac{\partial^2 Z}{\partial x^2}+\frac{\partial^2 Z}{\partial y^2}+\mu Z=0, \quad 0<x,y<1,
\end{equation}
\begin{equation} \label{8}
Z(0,y)=Z(1,y),\quad \frac{\partial Z(1,y)}{\partial x}=0,\quad Z(x,0)=Z(x,1)=0, \quad 0\leq x,y\leq1,
\end{equation}
where $\mu$ is the separation parameter. We present the solution in the form
\begin{equation}\label{10}
Z(x,y)=X(x)V(y).
\end{equation}
Substituting this expression into \eqref{7} and \eqref{8}, we obtain the following problems
\begin{equation}\label{11}
V{''}(y)+\lambda V(y)=0,\quad 0<y<1, \quad V(0)=V(1)=0,
\end{equation}
\begin{equation}\label{12}
X''(x)+\gamma X(x)=0,\quad 0<x<1, \quad X(0)=X(1),\quad X'(1)=0,
\end{equation}
where $\gamma=\mu-\lambda$. It is known the solutions of problem \eqref{11} have the form
$$\lambda_k=(\pi k)^2,\quad V_k(y)=\sqrt{2}\sin(\pi ky),\quad k=1,2,\ldots.$$
Here and in the following, we give the constants multiplying the eigenfunctions and associated functions from normalization conditions.

The eigenvalues and the corresponding eigenfunctions of problem \eqref{12} were mentioned in \cite{6}:
$$\gamma_m=(2\pi m)^2,\quad m=0,1,2,\ldots,\quad X_0=2,\quad X_m(x)=4\cos(2\pi mx),\quad m=1,2,\ldots.$$
Consequently, the eigenvalues and eigenfunctions of problem \eqref{7}, \eqref{8} with the representation \eqref{10} have the form
$$\mu_{m,k}=\gamma_m+\lambda_k=(2\pi m)^2+(\pi k)^2,\quad Z_{m,k}(x,y)=X_m(x)V_k(y), \quad m=0,1,2,\ldots, \quad k=1,2,\ldots.$$
Here, the set of eigenfunctions $Z_{m,k}(x,y)$ is not complete in the space $L_2(D_{xy})$; therefore, it is supplemented the set of eigenfunctions by the following associated functions $\overset{\sim}{Z}_{m,k}(x,y)$, $m,k=1,2\ldots$ in \cite{IM00} by following the lines of \cite{3}
$$\overset{\sim}{Z}_{m,k}=4(1-x)\sin(2\pi mx)\sqrt{2}\sin(\pi ky).$$
We denote the system of root functions of problem \eqref{7}, \eqref{8} as follows \cite{IM00}:
\begin{equation}\label{14}
Z_{0,k}(x,y)=X_0(x)V_k(y), \quad Z_{2m-1,k}(x,y)=X_m(x)V_k(y),\quad Z_{2m,k}(x,y)=\overset{\sim}{Z}_{m,k},\quad k,m=1,2,\ldots
\end{equation}
The adjoint problem of \eqref{7}, \eqref{8} is 
\begin{equation} \label{9}
\begin{split}
\partial^2 W/\partial x^2+\partial^2 W/\partial y^2+\mu W=0, \quad 0<x,y<1, \\
W(0,y)=0,\quad \partial W(0,y)/\partial x=\partial W(1,y)/\partial x, \quad W(x,0)=W(x,1)=0, \quad 0\leq x,y\leq 1.
\end{split}
\end{equation}
The root functions of the adjoint problem \eqref{9} have the form 
\begin{equation}\label{15}
\begin{split}
W_{0,k}(x,y)=xV_k(y),\quad W_{2m,k}(x,y)=\sin(2\pi mx)V_k(y),\\
 W_{2m-1,k}(x,y)=x\cos(2\pi mx)V_k(y),\quad m,k=1,2,\ldots.
\end{split}
\end{equation}
Here, the sequences \eqref{14}, \eqref{15} form a biorthonormal system of functions on $\overline{D}_{xy}$; i.e., for any admissible indices $m,k,l,$ and $p$, we have $(Z_{m,k},W_{lp})=1$ if $m=l$ and $k=p$ and $(Z_{m,k},W_{l,p})=0$ otherwise. And each of these sequence is a basis in the space $L_2(D_{xy})$ \cite{IM00}. Here, the inner product is given by
$$(\psi,\xi)=\Int_0^1\Int_0^1\psi(x,y)\xi(x,y)\diff x\diff y.$$ 

\subsection{Finite difference discretization}
We represent the derivatives in a differential equation with finite difference discretization both on a uniform and on a non-uniform grid in order to solve the differential equation numerically. 

We construct an explicit finite difference scheme for this problem. First, let us show this scheme on a uniform grid.
\subsubsection{Uniform grid}
Let us consider a uniform grid with increments $hx$ and $hy$ with respect to $x$ and $y$ respectively and time increment $ht$. Set
\begin{equation}
\begin{split}
&x_i=i\cdot hx, \quad i=0,1,\ldots,Nx, \quad hx\cdot Nx=1, \\
&y_j=j\cdot hy, \quad j=0,1,\ldots,Ny, \quad hy\cdot Ny=1, \\
&t_n=n\cdot ht, \quad n=0,1,\ldots,Nt, \quad ht\cdot Nt=T.
\end{split}
\nonumber
\end{equation}
Let $u_{ij}^n=u(x_i,y_j,t_n)$, then from \cite{GIM01} we take the derivatives as follows:
\begin{equation}
\begin{split}
\left.\frac{\partial u}{\partial t}\right|_{(x_i,y_j,t_n)}&=\frac{u_{ij}^{n+1}-u_{ij}^n}{ht}, \\ 
\left.\frac{\partial^2 u}{\partial x^2}\right|_{(x_i,y_j,t_n)}=\frac{u_{i+1j}^{n}-2u_{ij}^n+u_{i-1j}^n}{hx^2},& \qquad \left.\frac{\partial^2 u}{\partial y^2}\right|_{(x_i,y_j,t_n)}=\frac{u_{ij+1}^{n}-2u_{ij}^n+u_{ij-1}^n}{hy^2}.
\end{split}
\nonumber
\end{equation}
Now let's see the explicit finite difference scheme on a non-uniform grid.

\subsubsection{Non-uniform grid}
When we consider a non-uniform grid, even if we don't have fixed increments for the spatial axes like on a uniform grid but we have fixed increment $ht$ for time. Set 
\begin{equation}
\begin{split}
&x_i, \quad i=0,1,\ldots,Nx, \\
&y_j, \quad y=0,1,\ldots,Ny, \\
t_n=n\cdot ht, &\quad n=0,1,\ldots,Nt,\quad ht\cdot Nt=T.
\end{split}
\nonumber
\end{equation}
We take the first derivatives as
$$\left.\frac{\partial u}{\partial x}\right|_{(x_i,y_j,t_n)}=\frac{u_{ij}^{n}-u_{i-1j}^n}{x_i-x_{i-1}},\mbox{ for }i=1,\ldots,Nx,\quad \left.\frac{\partial u}{\partial t}\right|_{(x_i,y_j,t_n)}=\frac{u_{ij}^{n+1}-u_{ij}^n}{ht},$$
and we take the second order derivatives on the spatial axes as follows \cite{BS05}
\begin{equation}
\begin{split}
&\left.\frac{\partial^2 u}{\partial x^2}\right|_{(x_i,y_j,t_n)}=\frac{2u_{i-1j}^n}{(x_{i-1}-x_{i})(x_{i-1}-x_{i+1})}+\frac{2u_{ij}^n}{(x_{i}-x_{i-1})(x_{i}-x_{i+1})}+\frac{2u_{i+1j}^n}{(x_{i+1}-x_{i-1})(x_{i+1}-x_{i})}, \\
&\left.\frac{\partial^2 u}{\partial y^2}\right|_{(x_i,y_j,t_n)}=\frac{2u_{ij-1}^n}{(y_{j-1}-y_{j})(y_{j-1}-y_{j+1})}+\frac{2u_{ij}^n}{(y_{j}-y_{j-1})(y_{j}-y_{j+1})}+\frac{2u_{ij+1}^n}{(y_{j+1}-y_{j-1})(y_{j+1}-y_{j})}.
\end{split}
\nonumber
\end{equation}

\subsection{Numerical Integration}
We consider several numerical integration formulas both for uniform and non-uniform grids.
\subsubsection{On uniform grid}
Let 
\begin{equation}
\begin{split}
I&=\iint_{D_{xy}}u(x,y)\diff x\diff y,\\
x_i=ih,&~y_j=jk,\quad h=1/n,~k=1/m.
\end{split}
\nonumber
\end{equation}
\textit{The product trapezoidal rule} is
\begin{equation}
I\approx hk\sum_{i=1}^n\sum_{j=1}^m \frac{1}{4}(u(x_{i-1},y_{j-1})+u(x_{i-1},y_{j})+u(x_{i},y_{j-1})+u(x_{i},y_{j})).
\nonumber
\end{equation}
\textit{The product Simpson's rule} is
\begin{equation}
\begin{split}
I \approx \frac{hk}{9}\sum_{i=1}^{n/2}\sum_{j=1}^{m/2} [&u(x_{2i-2},y_{2j-2})+4u(x_{2i-1},y_{2j-2})+u(x_{2i},y_{2j-2})] \\
& +4[u(x_{2i-2},y_{2j-1})+4u(x_{2i-1},y_{2j-1})+u(x_{2i},y_{2j-1})] \\
& ~\\
&+ [u(x_{2i-2},y_{2j})+4u(x_{2i-1},y_{2j})+u(x_{2i},y_{2j})],
\end{split}
\nonumber
\end{equation}
where $m$ and $n$ must be even.

\subsubsection{On non-uniform grid}
We have product rules obtained from Gauss-Lobatto nodes and weights on the unit square. We get desired nodes by moving these nodes to $\overline{D}_{xy}$ from the unit square. Hence, if the increments of the nodes are not equal, then we can use the following product rule for numerical integration.
\begin{equation}
I\approx \frac{1}{4}\sum_{i=1}^n\sum_{j=1}^m A_{ij}u\left(\frac{1}{2}x_{i}+\frac{1}{2},\frac{1}{2}y_{j}+\frac{1}{2}\right),
\nonumber
\end{equation}
where $x_i$ and $y_j$ are nodes of Gauss-Lobatto nodes on unit square and $A_i$ and $A_j$ are corresponding weights to these nodes, respectively with $A_{ij}=A_iA_j$ \cite{AS65}.

\section{Well-posedness of inverse problem}
We have the following assumptions on $\varphi(x,y)$, $f(x,y,t)$ and $E(t)$.

\begin{theorem}[Existence and uniqueness] Under the conditions \label{teoexun}
\begin{equation}
\begin{split}
(A_1)_1 \quad & \varphi(x,y)\in C^{2,2}(\overline{D}_{xy}), \\ 
(A_1)_2 \quad & \varphi(0,y)=\varphi(1,y),\quad \varphi(x,0)=\varphi(x,1)=0, \quad \varphi_x(1,y)=0, \\ 
(A_1)_3 \quad & \varphi_{0,2k-1}\leq 0, \quad \varphi_{2m,2k-1}\leq 0, \quad m,k=1,2\ldots, \\ 
\nonumber
\end{split}
\end{equation}
\begin{equation}
\begin{split}
(A_2)_1 \quad & f(x,y,t)\in C(\overline{D}),~~f(x,y,t)\in C^{2,2}(\overline{D}_{xy}), ~~ \forall t\in [0,T] \\ 
(A_2)_2 \quad & f(0,y)=f(1,y),\quad f(x,0)=f(x,1)=0, \\ 
(A_2)_3 \quad & f_{2m,2k-1}(t)\geq 0, \quad \underset{0\leq t\leq T}{\min}f_{2m,2k-1}(t)\geq \underset{0\leq t\leq T}{\max}f_{2m,2k-1}(t)\left[1-e^{-(\pi(2k-1))^2T-(2\pi m)^2T}\right],\\
& \qquad\qquad\qquad\qquad\qquad\qquad\qquad\qquad\qquad\qquad\qquad\qquad\qquad\qquad~m=0,1,\ldots,~k=1,2,\ldots, \\ 
\nonumber
\end{split}
\end{equation}
\begin{equation}
\begin{split}
(A_3)_1 \quad & E(t)\in C^1[0,T], \\ 
(A_3)_2 \quad & E(0)=\Int_0^1\Int_0^1\varphi(x,y)\diff x\diff y, \\ 
(A_3)_3 \quad & E(t)>0, \quad E'(t)\leq 0, \quad \forall t\in [0,T], \\ 
\nonumber
\end{split}
\end{equation}
the inverse problem \eqref{1}-\eqref{4} has a unique solution for small $T$.
\end{theorem}
\begin{proof}
Since \eqref{14} is a basis in $L_2(\overline{D}_{xy})$, we present the solution of \eqref{1}-\eqref{3} in the following form for arbitrary $p(t)\in C[0,T]$:
\begin{equation}\label{u}
u(x,y,t)=\sum_{k=1}^\infty\left[\alpha_{0,k}(t)Z_{0,k}(x,y)+\sum_{m=1}^\infty \alpha_{2m-1,k}(t)Z_{2m-1,k}(x,y)+\sum_{m=1}^\infty \alpha_{2m,k}(t)Z_{2m,k}(x,y) \right],
\end{equation}
where
\begin{equation}\label{alf}
\begin{split}
&\alpha_{0,k}(t)=\varphi_{0,k}\cdot e^{-(\pi k)^2t-\int_0^tp(s)\diff s}+\int_0^t f_{0,k}(\tau) e^{-(\pi k)^2(t-\tau)-\int_\tau^t p(s)\diff s}\diff\tau, \\
& \alpha_{2m,k}(t)=\varphi_{2m,k} e^{-(2\pi m)^2t-(\pi k)^2t-\int_0^tp(s)\diff s}+\int_0^t f_{2m,k}(\tau) e^{-(2\pi m)^2(t-\tau)-(\pi k)^2(t-\tau)-\int_\tau^t p(s)\diff s}\diff\tau, \\
& \alpha_{2m-1,k}(t)=\left[\varphi_{2m-1,k}-4\pi m\cdot\varphi_{2m,k}t\right]e^{-(2\pi m)^2t-(\pi k)^2t-\int_0^tp(s)\diff s}\\
&\qquad\qquad\qquad\qquad+\int_0^t\left[f_{2m-1,k}(\tau)-4\pi m f_{2m,k}(\tau)(t-\tau)\right]e^{-(2\pi m)^2(t-\tau)-(\pi k)^2(t-\tau)-\int_\tau^tp(s)\diff s}\diff\tau. 
\end{split}
\nonumber
\end{equation} 
Here 
\begin{equation}
\begin{split}
&\varphi_{m,k}=\iint_{D_{xy}} \varphi(x,y)W_{m,k}(x,y)\diff x \diff y,\\
&f_{m,k}(t)=\iint_{D_{xy}}f(x,y,t)W_{m,k}(x,y)\diff x \diff y, \quad m=0,1,2,\ldots,~k=1,2,\ldots.
\end{split}
\nonumber
\end{equation}
Under conditions $A_1$ and $A_2$, the series \eqref{u}, its $t-$partial derivative, the $xx-$second order and $yy-$second order partial derivatives converge uniformly in $\overline{D}$  that their majorizing sums absolutely convergence. Thus $u(x,y,t)\in C^{2,2,1}(\overline{D})$. Since the conditions that the series \eqref{u} can be termwise differentiable by $t$ and $E(t)\in C^1[0,T]$ with $(A_3)_2$ the overdetermination condition \eqref{4} is equivalent to 
\begin{equation} \label{Et}
\Int_0^1\Int_0^1u_t(x,y,t)\diff x\diff y=E'(t).
\end{equation}
Therefore we have 
\begin{equation} \label{p}
P(p(t))=p(t)
\end{equation}
from the equations \eqref{u}-\eqref{Et} such that
\begin{equation}
\begin{split}
P(p(t))=\frac{1}{E(t)}&\left[-E'(t)+\sum_{k=1}^\infty\frac{4\sqrt{2}}{\pi(2k-1)}f_{0,2k-1}(t)-4\sqrt{2}\pi(2k-1)\right.\\
&\qquad\times\left[\varphi_{0,2k-1}e^{-(\pi(2k-1))^2t-\int_0^tp(s)\diff s}+\int_0^tf_{0,2k-1}(\tau)e^{-(\pi(2k-1))^2(t-\tau)-\int_\tau^tp(s)\diff s}\diff\tau\right]\\
&~+\sum_{k=1}^\infty\sum_{m=1}^\infty\frac{4\sqrt{2}}{\pi^2(2k-1)m}f_{2m,2k-1}(t)-\left(\frac{16\sqrt{2}m}{2k-1}+\frac{4\sqrt{2}(2k-1)}{m}\right) \\
&\qquad\qquad\qquad\times\left[\varphi_{2m,2k-1}e^{-(2\pi m)^2t-(\pi(2k-1))^2t-\int_0^tp(s)\diff s}\right.\\
&\qquad\qquad\qquad\qquad\quad~\left.+\int_0^tf_{2m,2k-1}(\tau)e^{-(2\pi m)^2(t-\tau)-(\pi(2k-1))^2(t-\tau)-\int_\tau^tp(s)\diff s}\diff\tau\right].
\end{split}
\nonumber
\end{equation}
Now let us show that $P$ is a contraction mapping in $C^{+}[0,T]$, for small $T$, where
\begin{equation}
C^{+}[0,T]=\{p(t)\in C[0,T]:p(t)\geq 0\}.
\nonumber
\end{equation}
Moreover, it is easy to see that
\begin{equation}
P:C^{+}[0,T]\rightarrow C^{+}[0,T]
\nonumber
\end{equation}
under conditions $(A_1)_3$, $(A_2)_3$ and $(A_3)_3$. For all $p_1(t),p_2(t)\in C^{+}[0,T]$,

\begin{equation}
\begin{split}
|P&(p_1(t))-P(p_2(t))|\\
\leq& \frac{1}{|E(t)|}\sum_{k=1}^\infty 4\sqrt{2}\pi (2k-1)\left[|\varphi_{0,2k-1}|\left|e^{-\int_0^tp_1(s)\diff s}-e^{-\int_0^tp_2(s)\diff s}\right|\right.\\
&\qquad\qquad\qquad\qquad\qquad\qquad\qquad\qquad\qquad\quad~~\left.+\int_0^t\left|f_{0,2k-1}(\tau)\right|\left|e^{-\int_\tau^tp_1(s)\diff s}-e^{-\int_\tau^tp_2(s)\diff s}\right|\diff\tau\right] \\
&+\frac{1}{|E(t)|}\sum_{k=1}^\infty\sum_{m=1}^\infty\left(\frac{16\sqrt{2}m}{2k-1}+\frac{4\sqrt{2}(2k-1)}{m}\right)\\
&\quad~\times\left[|\varphi_{2m,2k-1}|\left|e^{-\int_0^tp_1(s)\diff s}-e^{-\int_0^tp_2(s)\diff s}\right|+\int_0^t\left|f_{2m,2k-1}(\tau)\right|\left|e^{-\int_\tau^tp_1(s)\diff s}-e^{-\int_\tau^tp_2(s)\diff s}\right|\diff\tau\right].
\end{split}
\nonumber
\end{equation}
By using the mean value theorem we get
\begin{equation}
\left|e^{-\int_0^tp_1(s)\diff s}-e^{-\int_0^tp_2(s)\diff s}\right| \leq T\cdot\underset{0\leq t\leq T}{\max}|p_1(t)-p_2(t)|=T\|p_1-p_2\|_{C[0,T]}.
\nonumber
\end{equation}
The last inequality yields to the following existence and uniqueness results:
\begin{equation}
\|P(p_1)-P(p_2)\|_{C[0,T]}\leq \beta \|p_1-p_2\|_{C[0,T]},
\nonumber
\end{equation}
where 
\begin{equation}
\begin{split}
\beta=\frac{T}{\underset{0\leq t\leq T}{\min} |E(t)|}&\left(\sum_{k=1}^\infty 4\sqrt{2}\pi(2k-1)\left| \varphi_{0,2k-1}\right|+\int_0^T\sum_{k=1}^\infty 4\sqrt{2}\pi(2k-1)\left| f_{0,2k-1}(\tau)\right|\diff\tau \right. \\
& \qquad\qquad +\sum_{k=1}^\infty\sum_{m=1}^\infty\left(\frac{16\sqrt{2}m}{2k-1}+\frac{4\sqrt{2}(2k-1)}{m}\right)|\varphi_{2m,2k-1}|\\
& ~~\qquad\qquad\qquad\quad \left.+\int_0^T\sum_{k=1}^\infty\sum_{m=1}^\infty\left(\frac{16\sqrt{2}m}{2k-1}+\frac{4\sqrt{2}(2k-1)}{m}\right)|f_{2m,2k-1}(\tau)|\diff \tau \right).
\end{split}
\nonumber
\end{equation}
In the case $\beta<1$ the map $P$ is contraction map in $C^{+}[0,T]$. Obviously, this inequality is satisfied for small $T$. Hence, $P$ has unique fixed point by Banach fixed point theorem. \qedwhite

\end{proof}

The following result on continuously dependence on the data of the solution of inverse problem \eqref{1}-\eqref{4} holds.

\begin{theorem}[Stability]
Under assumptions $(A_1)-(A_3)$, the solution $(p,u)$ depends continuously upon the data. 
\end{theorem}
\begin{proof}
Let $\Phi=\{\varphi,E,f\}$ and $\overline{\Phi}=\{\overline{\varphi},\overline{E},\overline{f}\}$ be two sets of data, which satisfy the conditions $(A_1)-(A_3)$. Denote $\|\Phi\|=\|\varphi\|_{C^{2,2}(\overline{D}_{xy})}+\|E\|_{C^1[0,T]}+\|f\|_{C^{2,2,0}(\overline{D})}$. Suppose that there exist positive constants $M_1$ and $M_2$ such that
$$0<M_1\leq |E|,\quad 0<M_1\leq |\overline{E}|,\qquad \|\Phi\|\leq M_2, \quad \|\overline{\Phi}\|\leq M_2.$$
Let $(p,u)$ and $(\overline{p},\overline{u})$ be the solutions of inverse problem \eqref{1}-\eqref{4} corresponding the data $\Phi$ and $\overline{\Phi}$, respectively. According to \eqref{p}
\begin{equation}
\begin{split}
p(t)=\frac{1}{E(t)}&\left[-E'(t)+\sum_{k=1}^\infty\frac{4\sqrt{2}}{\pi(2k-1)}f_{0,2k-1}(t)-4\sqrt{2}\pi(2k-1)\right.\\
&\qquad~~\times\left[\varphi_{0,2k-1}e^{-(\pi(2k-1))^2t-\int_0^tp(s)\diff s}+\int_0^tf_{0,2k-1}(\tau)e^{-(\pi(2k-1))^2(t-\tau)-\int_\tau^tp(s)\diff s}\diff\tau\right]\\
&~+\sum_{k=1}^\infty\sum_{m=1}^\infty\frac{4\sqrt{2}}{\pi^2(2k-1)m}f_{2m,2k-1}(t)-\left(\frac{16\sqrt{2}m}{2k-1}+\frac{4\sqrt{2}(2k-1)}{m}\right) \\
&\qquad\qquad\qquad\times\left[\varphi_{2m,2k-1}e^{-(2\pi m)^2t-(\pi(2k-1))^2t-\int_0^tp(s)\diff s}\right.\\
&\qquad\qquad\qquad\qquad\qquad~~\left.+\int_0^tf_{2m,2k-1}(\tau)e^{-(2\pi m)^2(t-\tau)-(\pi(2k-1))^2(t-\tau)-\int_\tau^tp(s)\diff s}\diff\tau\right],
\end{split}
\nonumber
\end{equation}
\begin{equation}
\begin{split}
\overline{p}(t)=\frac{1}{\overline{E}(t)}&\left[-\overline{E}'(t)+\sum_{k=1}^\infty\frac{4\sqrt{2}}{\pi(2k-1)}\overline{f}_{0,2k-1}(t)-4\sqrt{2}\pi(2k-1)\right.\\
&\qquad~~\times\left[\overline{\varphi}_{0,2k-1}e^{-(\pi(2k-1))^2t-\int_0^t\overline{p}(s)\diff s}+\int_0^t\overline{f}_{0,2k-1}(\tau)e^{-(\pi(2k-1))^2(t-\tau)-\int_\tau^t\overline{p}(s)\diff s}\diff\tau\right]\\
&~+\sum_{k=1}^\infty\sum_{m=1}^\infty\frac{4\sqrt{2}}{\pi^2(2k-1)m}\overline{f}_{2m,2k-1}(t)-\left(\frac{16\sqrt{2}m}{2k-1}+\frac{4\sqrt{2}(2k-1)}{m}\right) \\
&\qquad\qquad\qquad\times\left[\overline{\varphi}_{2m,2k-1}e^{-(2\pi m)^2t-(\pi(2k-1))^2t-\int_0^t\overline{p}(s)\diff s}\right.\\
&\qquad\qquad\qquad\qquad\qquad~~\left.+\int_0^t\overline{f}_{2m,2k-1}(\tau)e^{-(2\pi m)^2(t-\tau)-(\pi(2k-1))^2(t-\tau)-\int_\tau^t\overline{p}(s)\diff s}\diff\tau\right].
\end{split}
\nonumber
\end{equation}
Let us estimate the difference $p-\overline{p}$ at first. It is easy to compute the followings
\begin{equation}
\left|\frac{E'(t)}{E(t)}-\frac{\overline{E}'(t)}{\overline{E(t)}}\right|\leq M_3\|E-\overline{E}\|_{C^1[0,T]},
\nonumber
\end{equation}
\begin{equation}
\begin{split}
\left|\sum_{k=1}^\infty \frac{1}{2k-1}\left(\frac{f_{0,2k-1}(t)}{E(t)}-\frac{\overline{f}_{0,2k-1}(t)}{\overline{E}(t)}\right)\right|\leq M_4\|f-\overline{f}\|_{C^{2,2,0}(\overline{D})}+M_5\|E-\overline{E}\|_{C^1[0,T]},
\end{split}
\nonumber
\end{equation}
\begin{equation}
\begin{split}
&\left|\sum_{k=1}^\infty (2k-1)\left(\frac{1}{E(t)}\varphi_{0,2k-1}e^{-(\pi(2k-1))^2t-\int_{0}^{t}p(s)\diff s}-\frac{1}{\overline{E}(t)}\overline{\varphi}_{0,2k-1}e^{-(\pi(2k-1))^2t\int_0^t \overline{p}(s)\diff s}\right)\right|\\
&\qquad\qquad\qquad\qquad\qquad\qquad\leq M_6\|\varphi-\overline{\varphi}\|_{C^{2,2}(\overline{D}_{xy})}+M_7T\|p-\overline{p}\|_{C[0,T]}+M_8\|E-\overline{E}\|_{C^1[0,T]},
\end{split}
\nonumber
\end{equation}
\begin{equation}
\begin{split}
&\left|\sum_{k=1}^\infty (2k-1)\left(\frac{1}{E(t)}\int_0^t f_{0,2k-1}(\tau)e^{-(\pi(2k-1))^2(t-\tau)-\int_\tau^t p(s)\diff s}\diff\tau \right.\right. \\
&\qquad\qquad\qquad\qquad\qquad\qquad\qquad\qquad\left.\left.-\frac{1}{\overline{E}(t)}\int_0^t \overline{f}_{0,2k-1}(\tau)e^{-(\pi(2k-1))^2(t-\tau)-\int_\tau^t\overline{p}(s)\diff s}\diff\tau \right)\right|\\
&\qquad\qquad\qquad\qquad\qquad\leq M_9 T \|f-\overline{f}\|_{C^{2,2,0}(\overline{D})}+ M_{10}T^2\|p-\overline{p}\|_{C[0,T]}+M_{11} T \|E-\overline{E}\|_{C^1[0,T]},
\end{split}
\nonumber
\end{equation}
\begin{equation}
\left|\sum_{k=1}^\infty \sum_{m=1}^\infty \frac{1}{(2k-1)m}\left(\frac{f_{2m,2k-1}(t)}{E(t)}-\frac{\overline{f}_{2m,2k-1}(t)}{\overline{E}(t)}\right)\right| \leq M_{12} \|f-\overline{f}\|_{C^{2,2,0}(\overline{D})}+ M_{13} \|E-\overline{E}\|_{C^1[0,T]},
\nonumber
\end{equation}
\begin{equation}
\begin{split}
&\left|\sum_{k=1}^\infty \sum_{m=1}^\infty \left(\frac{4m}{2k-1}+\frac{2k-1}{m}\right)\left(\frac{\varphi_{2m,2k-1}}{E(t)}e^{-(2\pi m)^2t-(\pi(2k-1))^2t-\int_0^t p(s)\diff s}\right.\right.\\
&\qquad\qquad\qquad\qquad\qquad\qquad\qquad\qquad\left.\left.-\frac{\overline{\varphi}_{2m,2k-1}}{\overline{E}(t)}e^{-(2\pi m)^2t-(\pi(2k-1))^2t-\int_0^t \overline{p}(s)\diff s}\right)\right| \\
& \qquad\qquad\qquad\qquad\qquad\qquad\qquad\leq M_{14}\|\varphi-\overline{\varphi}\|_{C^{2,2}(\overline{D}_{xy})}+M_{15}T\|p-\overline{p}\|_{C[0,T]}+M_{16}\|E-\overline{E}\|_{C^1[0,T]},
\end{split}
\nonumber
\end{equation}
\begin{equation}
\begin{split}
&\left|\sum_{k=1}^\infty \sum_{m=1}^\infty \left(\frac{4m}{2k-1}+\frac{2k-1}{m}\right)\left(\frac{1}{E(t)}\int_0^t f_{2m,2k-1}(\tau) e^{-(2\pi m)^2(t-\tau)-(\pi(2k-1))^2(t-\tau)-\int_\tau^t p(s)\diff s}\diff\tau \right.\right. \\
&\qquad\qquad\qquad\qquad\qquad\qquad\left.\left.+\frac{1}{\overline{E}(t)}\int_0^t \overline{f}_{2m,2k-1}(\tau) e^{-(2\pi m)^2(t-\tau)-(\pi(2k-1))^2(t-\tau)-\int_\tau^t \overline{p}(s)\diff s}\diff\tau\right)\right| \\
&\qquad\qquad\qquad\quad\qquad\qquad\leq M_{17}T\|f-\overline{f}\|_{C^{2,2,0}(\overline{D})}+M_{18}T^2\|p-\overline{p}\|_{C[0,T]}+M_{19}T\|E-\overline{E}\|_{C^1[0,T]},
\end{split}
\nonumber
\end{equation}
where $M_i$, $i=3,4,\ldots,19$ are constants that are determined from $M_1$ and $M_2$.
Then we get
$$(1-M_{20})\|p-\overline{p}\|_{C[0,T]}\leq M_{21}\left(\|\varphi-\overline{\varphi}\|_{C^{2,2}(\overline{D}_{xy})}+\|E-\overline{E}\|_{C^1[0,T]}+\|f-\overline{f}\|_{C^{2,2,0}(\overline{D})}\right),$$
where $M_{20}=4\sqrt{2}\pi T(M_7+TM_{10})+4\sqrt{2}T(M_{15}+TM_{18})$ and $M_{21}=\max\{4\sqrt{2}/\pi M_4+4\sqrt{2}\pi TM_9+4\sqrt{2}/\pi^2M_{12}+4\sqrt{2}TM_{17},M_3+4\sqrt{2}/\pi M_5+4\sqrt{2}\pi(M_8+TM_{11})+4\sqrt{2}/\pi^2M_{13}+4\sqrt{2}(M_{16}+TM_{19}),$ $4\sqrt{2}\pi M_6+4\sqrt{2}M_{14}\}$. The inequality $M_{20}< 1$ holds for small $T$. Finally, we have
\begin{equation}
\|p-\overline{p}\|_{C[0,T]}\leq M_{22}\|\Phi-\overline{\Phi}\|,\qquad M_{22}=\frac{M_{21}}{1-M_{20}}.
\nonumber
\end{equation}
Similarly, we obtain the estimate the difference $u-\overline{u}$ from \eqref{u}:
\begin{equation}
\|u-\overline{u}\|_{C(\overline{D})}\leq M_{23} \|\Phi-\overline{\Phi}\|.
\nonumber
\end{equation} 
\qedwhite
\end{proof}

Now, let's see how we implement the numerical solution of the inverse problem in two different methods.

\section{Numerical methods for inverse problem}
We will consider the examples of numerical solution of the inverse problem \eqref{1}-\eqref{4}. For the convenience of discussion of the numerical method, we will rewrite the equations as follows:
\begin{equation} \label{v1}
\partial v/\partial t=\partial^2v/\partial x^2+\partial^2v/\partial y^2+r(t)f(x,y,t),
\end{equation}
\begin{equation} \label{v2}
v(x,y,0)=\varphi(x,y),\quad 0\leq x,y\leq 1,
\end{equation}
\begin{equation} \label{v3}
\begin{split}
&v(0,y,t)=v(1,y,t),\quad v_x(1,y,t)=0, \quad 0\leq y\leq1,~ 0\leq t\leq T,\\
&v(x,0,t)=v(x,1,t)=0,\quad 0\leq x\leq1,\quad 0\leq t\leq T,
\end{split}
\end{equation}
\begin{equation} \label{v4}
\Int_0^1\Int_0^1 v(x,y,t)\diff x\diff y=E(t)r(t),
\end{equation}
by using the transformations 
\begin{equation}
r(t)=e^{\int_0^t p(s)\diff s},
\nonumber 
\end{equation}
\begin{equation}
v(x,y,t)=u(x,y,t)r(t).
\nonumber
\end{equation}

\subsection{Uniform finite difference method}
We consider the following explicit finite difference discretization of the problem \eqref{v1}-\eqref{v4} on a uniform grid.
\begin{equation}
\begin{split}
\frac{v_{ij}^{n+1}-v_{ij}^{n}}{ht}=\frac{v_{i+1j}^n-2v_{ij}^n+v_{i-1j}^n}{hx^2}+\frac{v_{ij+1}^n-2v_{ij}^n+v_{ij-1}^n}{hy^2}&+r^nf_{ij}^n, \\
\Rightarrow v_{ij}^{n+1}=\frac{ht}{hx^2}v_{i+1j}^n+\frac{ht}{hy^2}v_{ij+1}^n+\left(1-2\frac{ht}{hx^2}-2\frac{ht}{hy^2} \right)v_{ij}^n+\frac{ht}{hx^2}v_{i-1j}^n+&\frac{ht}{hy^2}v_{ij-1}^n+ht\cdot r^nf_{ij}^n
\end{split}
\nonumber
\end{equation}
with the initial condition
\begin{equation}
\varphi_{ij}=v_{ij}^0,
\nonumber
\end{equation}
nonlocal boundary conditions
\begin{equation}
\begin{split}
&v_{0j}^n=v_{Nxj}^n, \qquad \frac{v_{Nxj}^n-v_{Nx-1j}^n}{hx}=0 \Rightarrow v_{Nxj}^n=v_{Nx-1j}^n, \\
&v_{i0}^n=v_{iNy}^n=0, 
\end{split}
\nonumber
\end{equation}
and overdetermination condition
\begin{equation} \label{runi}
r^n=\frac{1}{E^n}\Int_0^1\Int_0^1 v(x,y,t_n)\diff x\diff y,
\end{equation}
where
\begin{equation}
\begin{split}
v_{ij}^n=v(x_i,y_j,t_n), \quad r^n=r(t_n),\quad E^n=E(t_n),\quad f_{ij}^n=f(x_i,y_j,t_n), \quad n=0,\ldots, Nt,\qquad\qquad \\
\quad i=0,\ldots,Nx,~ j=0,\ldots,Ny.
\end{split}
\nonumber
\end{equation}
When we approximate $\Int_0^1\Int_0^1 v(x,y,t)\diff x\diff y$ by the product trapezoidal rule
\begin{equation}
\Int_0^1\Int_0^1 v(x,y,t)\diff x\diff y\approx hx\cdot hy\sum_{i=1}^{Nx}\sum_{j=1}^{Ny}\frac{1}{4}\left(v(x_{i-1},y_{j-1},t)+v(x_{i-1},y_{j},t)+v(x_{i},y_{j-1},t)+v(x_{i},y_{j},t)\right),
\nonumber
\end{equation}
where $t\in[0,T]$. And similarly by the product Simpson's rule
\begin{equation}
\begin{split}
\Int_0^1\Int_0^1 v(x,y,t)\diff x\diff y\approx \frac{hx\cdot hy}{9}\sum_{i=1}^{Nx/2}\sum_{j=1}^{Ny/2}&\left[v(x_{2i-2},y_{2j-2},t)+4v(x_{2i-1},y_{2j-2},t)+v(x_{2i},y_{2j-2},t)\right] \\
&~+4\left[v(x_{2i-2},y_{2j-1},t)+4v(x_{2i-1},y_{2j-1},t)+v(x_{2i},y_{2j-1},t)\right]\\
&~~\\
&~+\left[v(x_{2i-2},y_{2j},t)+4v(x_{2i-1},y_{2j},t)+v(x_{2i},y_{2j},t)\right],
\end{split}
\nonumber
\end{equation}
where $Nx$ and $Ny$ must be even.

\subsection{Non-uniform finite difference method}
Let us consider the same problem on a non-uniform grid.
\begin{equation}
\begin{split}
\frac{v_{ij}^{n+1}-v_{ij}^{n}}{ht}&=2\left(\frac{v_{i-1j}^n}{(x_{i-1}-x_i)(x_{i-1}-x_{i+1})}+\frac{v_{ij}^n}{(x_{i}-x_{i-1})(x_{i}-x_{i+1})}+\frac{v_{i+1j}^n}{(x_{i+1}-x_{i-1})(x_{i+1}-x_{i})}\right) \\
&+2\left(\frac{v_{ij-1}^n}{(y_{j-1}-y_j)(y_{j-1}-y_{j+1})}+\frac{v_{ij}^n}{(y_{j}-y_{j-1})(y_{j}-y_{j+1})}+\frac{v_{ij+1}^n}{(y_{j+1}-y_{j-1})(y_{j+1}-y_{j})}\right)\\
&+r^nf_{ij}^n,\\
\Rightarrow v_{ij}^{n+1}=&\frac{2ht}{(x_{i-1}-x_{i})(x_{i-1}-x_{i+1})}v_{i-1j}^n+\frac{2ht}{(y_{j-1}-y_{j})(y_{j-1}-y_{j+1})}v_{ij-1}^n\qquad\\
&+\left(1+\frac{2ht}{(x_{i}-x_{i-1})(x_{i}-x_{i+1})}+\frac{2ht}{(y_{j}-y_{j-1})(y_{j}-y_{j+1})}\right)v_{ij}^n\\
&+\frac{2ht}{(x_{i+1}-x_{i-1})(x_{i+1}-x_{i})}v_{i+1j}^n+\frac{2ht}{(y_{j+1}-y_{j-1})(y_{j+1}-y_{j})}v_{ij+1}^n+ht\cdot r^nf_{ij}^n,
\end{split}
\nonumber
\end{equation}
with the initial condition
\begin{equation}
\varphi_{ij}=v_{ij}^0,
\nonumber
\end{equation}
nonlocal boundary conditions
\begin{equation}
\begin{split}
&v_{0j}^n=v_{Nxj}^n, \qquad \frac{v_{Nxj}^n-v_{Nx-1j}^n}{x_{Nx}-x_{Nx-1}}=0 \Rightarrow v_{Nxj}^n=v_{Nx-1j}^n, \\
&v_{i0}^n=v_{iNy}^n=0, 
\end{split}
\nonumber
\end{equation}
and overdetermination condition
\begin{equation} \label{rnon}
r^n=\frac{1}{E^n}\Int_0^1\Int_0^1 v(x,y,t_n)\diff x\diff y,
\end{equation}
where
\begin{equation}
\begin{split}
&v_{ij}^n=v(x_i,y_j,t_n), \quad r^n=r(t_n),\quad E^n=E(t_n),\quad f_{ij}^n=f(x_i,y_j,t_n), \quad n=0,\ldots, Nt, \\
&\qquad\qquad\qquad\qquad\qquad\qquad\qquad\qquad\qquad 0=x_0<x_1<\cdots<x_{Nx}=1, ~0=y_0<y_1<\cdots<y_{Ny}=1.
\end{split}
\nonumber
\end{equation}
When we approximate $\Int_0^1\Int_0^1 v(x,y,t)\diff x\diff y$ by the product rules on the square using Gauss-Lobatto nodes which are transformed from the unit square and the weights, 
\begin{equation}
\Int_0^1\Int_0^1 v(x,y,t)\diff x\diff y\approx \frac{1}{4}\sum_{i=0}^{Nx}\sum_{j=0}^{Ny}A_{ij}\cdot v(x_i,y_j,t),
\nonumber
\end{equation}
where $0\leq x_1<\cdots<x_{Nx}\leq 1$ and $0\leq y_1<\cdots<y_{Ny}\leq1$, and $A_{ij}=A_iA_j$ are products of the weights $A_i$ and $A_j$.

Now we consider two examples and see the advantage of the using the product Gauss-Lobatto rules for numerical integration. At first let consider the example which does not satisfy the conditions $(A_1)-(A_3)$.

\begin{example}
Consider the inverse problem \eqref{1}-\eqref{4}, with
\begin{equation}
\begin{split}
&f(x,y,t)=e^{tx(x-1)^2}\left(9y^2t^2x^4-24y^2t^2x^3+22y^2t^2x^2-8y^2t^2x-6ytx+6y^2tx-9yt^2x^4\right.\\
&\qquad\qquad\qquad\qquad\qquad\left.+24yt^2x^3-22yt^2x^2+8yt^2x+yx^3-2yx^2-y^2x^3+2y^2x^2-y^2x+2+yx\right), \\
&\qquad\qquad\qquad\qquad\varphi(x,y)=y(1-y),\qquad E(t)=\frac{1}{6}\int_0^1 e^{tx^3-2tx^2+tx}\diff x,\qquad T=1.
\end{split}
\nonumber
\end{equation}
It is easy to check
\begin{equation}
\{p(t),u(x,y,t)\}=\{t^2-4t,~y(1-y)e^{tx(x-1)^2}\}.
\nonumber
\end{equation}
Then the problem \eqref{v1}-\eqref{v4} will be
\begin{equation}
\begin{split}
v_t=v_{xx}+v_{yy}+r(t)&e^{tx(x-1)^2}\left(9y^2t^2x^4-24y^2t^2x^3+22y^2t^2x^2-8y^2t^2x-6ytx+6y^2tx-9yt^2x^4\right.\\
&\quad~\left.+24yt^2x^3-22yt^2x^2+8yt^2x+yx^3-2yx^2-y^2x^3+2y^2x^2-y^2x+2+yx\right), \\
&\qquad\qquad\qquad\qquad\qquad\qquad\qquad0<x,y<1,\quad 0<t<1, \\
&v(x,y,0)=y(1-y), \quad 0\leq y\leq 1, \\
&v(0,y,t)=v(1,y,t), \quad 0\leq y,t\leq 1, \\
v(x,&0,t)=v(x,1,t)=0, \quad 0\leq x,t \leq 1, \\
&v_x(1,y,t)=0, \quad 0\leq y,t \leq 1, \\
\Int_0^1\Int_0^1 & v(x,y,t)\diff x\diff y=r(t)\frac{1}{6}\int_0^1 e^{tx^3-2tx^2+tx}\diff x,
\end{split}
\nonumber
\end{equation}
where the exact solution of $r(t)$ is
$$r(t)=\exp\left(\frac{t^3}{3}-2t^2\right).$$
We use the explicit finite difference method to solve the problem for finding the values of $u$, and then use \eqref{runi} and Trapezoidal rule and Simpson's rule on uniform grids; and \eqref{rnon} and product rule with Gauss-Lobatto nodes and weights to find $p(t)$ approximately. We get the following results on Table 1 and give the figures (Figure 1 - Figure 6) for $Nx=Ny=26$ and $Nt=2700$ on uniform grid, $Nt=147000$ on non-uniform grid (we choose these distinct $Nt$s' because of the convergence of explicit finite difference method): \\

\begin{table}[ht]
\begin{center}
\begin{tabular}{ c | c | c | c |}
\cline{2-4}
& {Error (Trapezoidal)} & {Error (Simpson's)} & {Error (Non-uniform)} \\
\hline
 \multicolumn{1}{|c|}{$u(x,y,t)$} & 0.0017 & 0.0012 & 1.1715e-04 \\ 
 \hline
  \multicolumn{1}{|c|}{$p(t)$} & 0.0609 & 0.0724 & 0.0059 \\
\hline
\end{tabular}
\caption{Errors of $u(x,y,t)$ for $t=1$ and $p(t)$.}
\end{center}
\end{table}

\begin{figure}[p]
\includegraphics[scale=0.39]{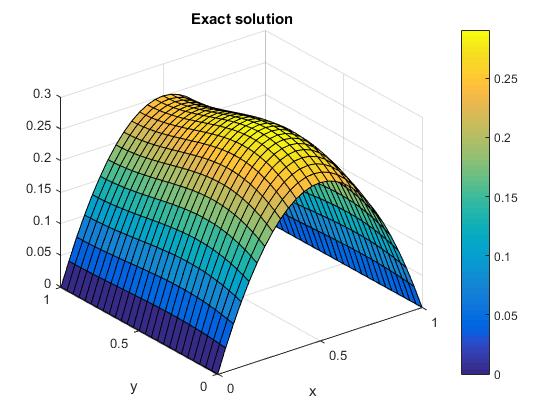}
\includegraphics[scale=0.39]{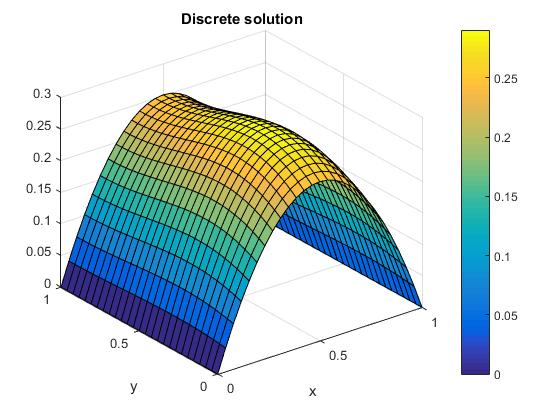}
\caption{Trapezoidal rule, exact and discrete solution of $u(x,y,t)$ for $t=1$.}
\end{figure}

\begin{figure}[p]
\begin{center}
\includegraphics[scale=0.39]{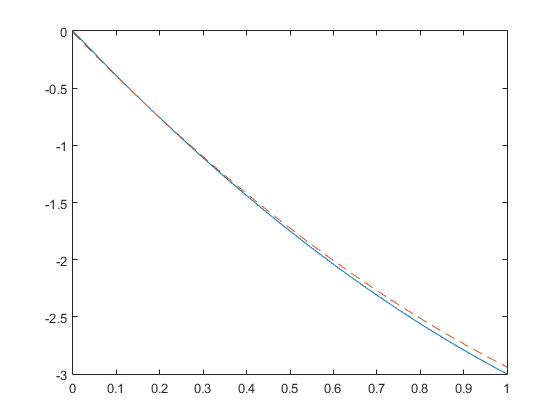}
\caption{Trapezoidal rule, exact and approximate solution of $p(t)$.}
\end{center}
\end{figure}

\begin{figure}[p]
\includegraphics[scale=0.39]{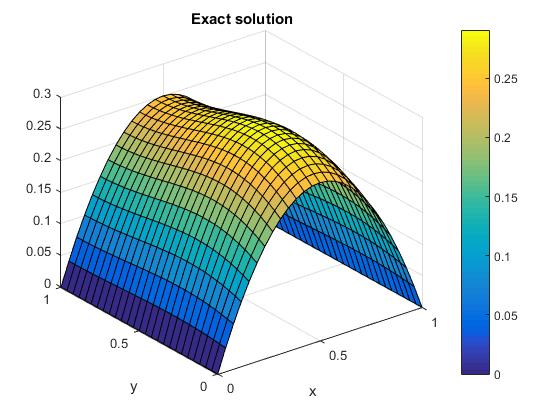}
\includegraphics[scale=0.39]{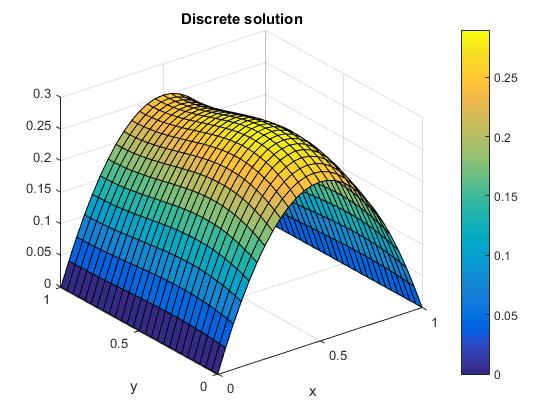}
\caption{Simpson's rule, exact and discrete solution of $u(x,y,t)$ for $t=1$.}
\end{figure}

\begin{figure}[p]
\begin{center}
\includegraphics[scale=0.39]{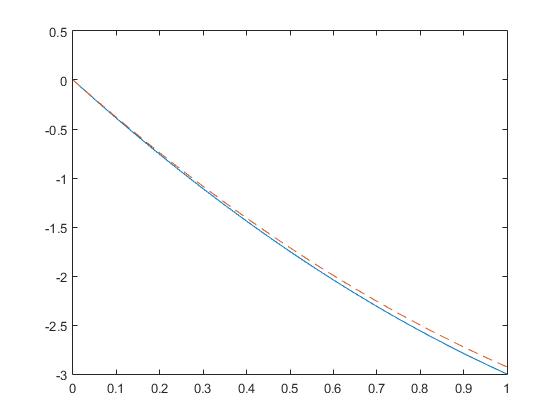}
\caption{Simpson's rule, exact and approximate solution of $p(t)$.}
\end{center}
\end{figure}

\begin{figure}[p]
\includegraphics[scale=0.39]{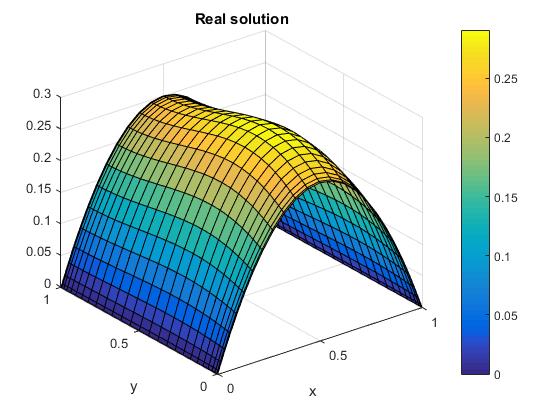}
\includegraphics[scale=0.39]{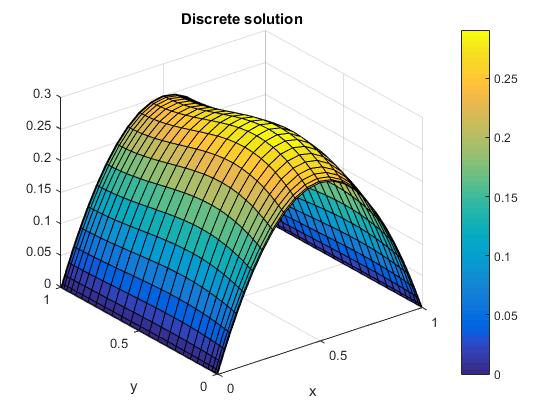}
\caption{Product Gauss-Lobatto rule, exact and discrete solution of $u(x,y,t)$ for $t=1$.}
\end{figure}

\begin{figure}[p]
\begin{center}
\includegraphics[scale=0.39]{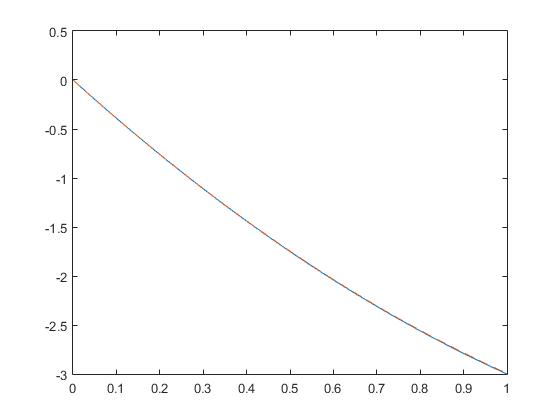}
\caption{Product Gauss-Lobatto rule, exact and approximate solution of $p(t)$.}
\end{center}
\end{figure}
\end{example}

\newpage 
Now let's consider an example which satisfies the conditions $(A_1)-(A_3)$.
\begin{example}
Consider the inverse problem \eqref{1}-\eqref{4}, with
\begin{equation}
\begin{split}
&f(x,y,t)=-\sin(\pi y)e^{-\pi^2t}\sin(\pi x) \left(9\pi^2\cos^2(\pi x) -3\pi^2-e^{\pi^2 t}+e^{\pi^2 t}\cos^2(\pi x)\right), \\
&\qquad\quad \varphi(x,y)=\sin(\pi y)\sin^3(\pi x),\qquad E(t)=\frac{8}{3\pi^2} e^{-\pi^2t},\qquad T=\frac{1}{3}.
\end{split}
\nonumber
\end{equation}
It is easy to check
\begin{equation}
\{p(t),u(x,y,t)\}=\{e^{\pi^2t},~\sin(\pi y)e^{-\pi^2t}\sin^3(\pi x)\}.
\nonumber
\end{equation}
Then the problem \eqref{v1}-\eqref{v4} will be
\begin{equation}
\begin{split}
v_t=v_{xx}+v_{yy}+r(t)&\left(-\sin(\pi y)e^{-\pi^2t}\sin(\pi x) \left(9\pi^2\cos^2(\pi x) -3\pi^2-e^{\pi^2 t}+e^{\pi^2 t}\cos^2(\pi x)\right)\right), \\
&\qquad\qquad\qquad\qquad\qquad\qquad\qquad\qquad0<x,y<1,\quad 0<t<\frac{1}{4}, \\
& v(x,y,0)=\sin(\pi y)\sin^3(\pi x), \quad 0\leq x,y\leq 1, \\
& v(0,y,t)=v(1,y,t), \quad  0\leq y\leq 1,~0\leq t\leq \frac{1}{4}, \\
& v(x,0,t)=v(x,1,t)=0, \quad 0\leq x \leq 1,~0\leq t\leq \frac{1}{4} \\
&\qquad v_x(1,y,t)=0, \quad 0\leq y\leq 1,~0\leq t\leq \frac{1}{4} \\
&\qquad \Int_0^1\Int_0^1  v(x,y,t)\diff x\diff y=r(t)\frac{8}{3\pi^2} e^{-\pi^2t},
\end{split}
\nonumber
\end{equation}
where the exact solution of $r(t)$ is
$$r(t)=\exp\left(\frac{-1+e^{\pi^2t}}{\pi^2}\right).$$
Actually, this example does satisfy the $(A_1)-(A_3)$ conditions with the exceptional $\varphi_{0,1}={\sqrt{2}}/{3\pi}>0$. This value must be non-positive, but we put this condition in order to make $p(t)$ non-negative. Since $p(t)$ is positive even if $\varphi_{0,1}\not\leq 0$ we can call this problem is available to the Theorem \ref{teoexun}. Now, we give the following results on Table 2 and several figures (Figure 7 - Figure 12) again for $Nx=Ny=26$ and $Nt=900$ on uniform grid, $Nt=49000$ on non-uniform grid (we choose these distinct $Nt$s' because of the convergence of explicit finite difference method): \\

\begin{table}[ht]
\begin{center}
\begin{tabular}{ c | c | c | c |}
\cline{2-4}
& {Error (Trapezoidal)} & {Error (Simpson's)} & {Error (Non-uniform)} \\
\hline
 \multicolumn{1}{|c|}{$u(x,y,t)$} & 4.1218e-04 & 3.8541e-04 & 1.9697e-04 \\ 
 \hline
  \multicolumn{1}{|c|}{$p(t)$} & 0.3114 & 0.2113 & 0.0196 \\
\hline
\end{tabular}
\caption{Errors of $u(x,y,t)$ for $t=\frac{1}{3}$ and $p(t)$.}
\end{center}
\end{table}

\begin{figure}[p]
\includegraphics[scale=0.39]{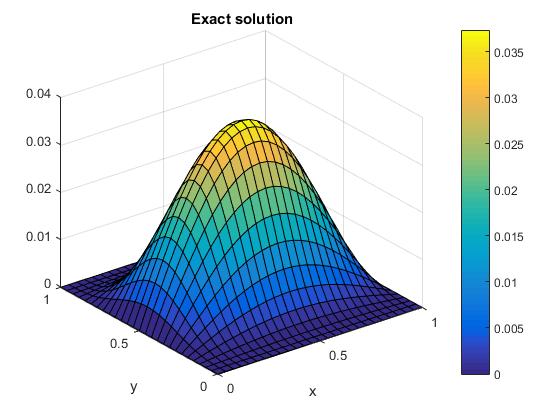}
\includegraphics[scale=0.39]{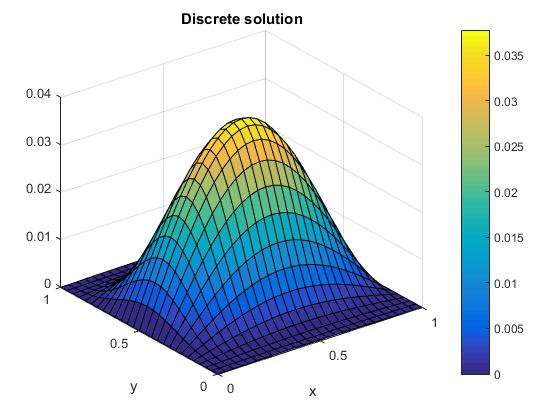}
\caption{Trapezoidal rule, exact and discrete solution of $u(x,y,t)$ for $t=\frac{1}{3}$.}
\end{figure}

\begin{figure}[p]
\begin{center}
\includegraphics[scale=0.39]{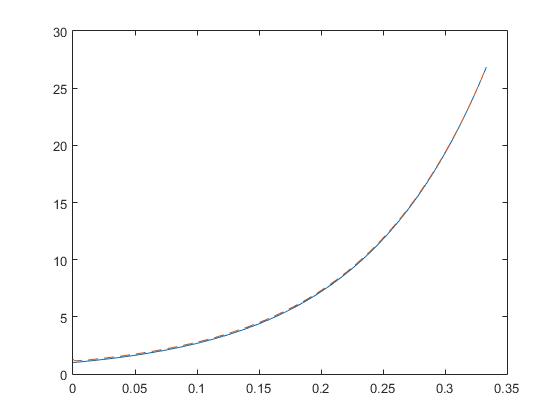}
\caption{Trapezoidal rule, exact and approximate solution of $p(t)$.}
\end{center}
\end{figure}

\begin{figure}[p]
\includegraphics[scale=0.39]{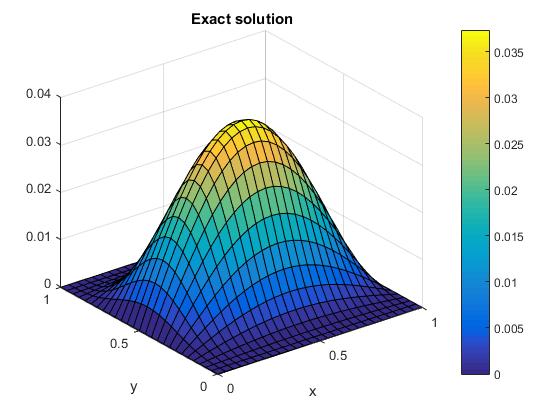}
\includegraphics[scale=0.39]{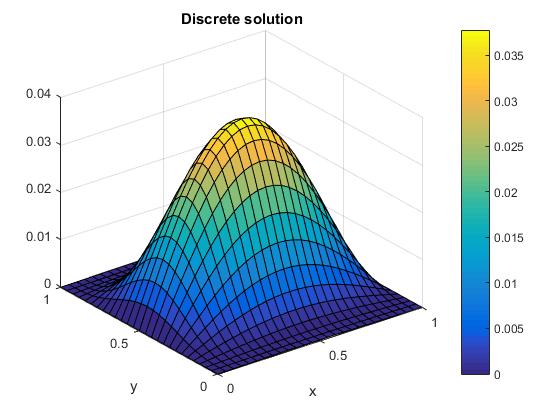}
\caption{Simpson's rule, exact and discrete solution of $u(x,y,t)$ for $t=\frac{1}{3}$.}
\end{figure}

\begin{figure}[p]
\begin{center}
\includegraphics[scale=0.39]{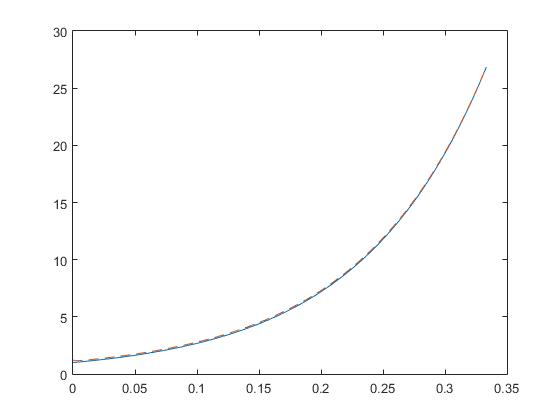}
\caption{Simpson's rule, exact and approximate solution of $p(t)$.}
\end{center}
\end{figure}

\begin{figure}[p]
\includegraphics[scale=0.39]{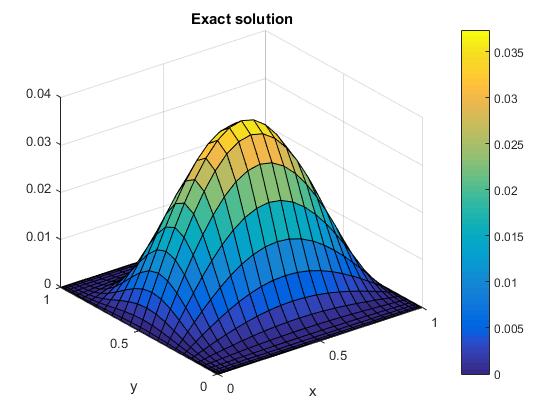}
\includegraphics[scale=0.39]{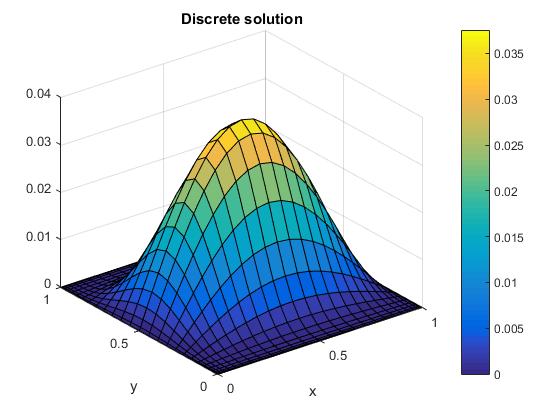}
\caption{Product Gauss-Lobatto rule, exact and discrete solution of $u(x,y,t)$ for $t=\frac{1}{3}$.}
\end{figure}

\begin{figure}[p]
\begin{center}
\includegraphics[scale=0.39]{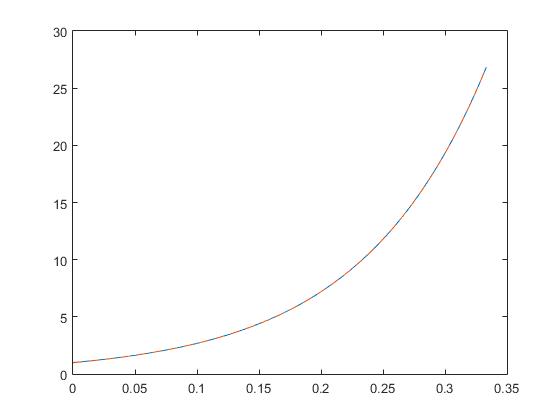}
\caption{Product Gauss-Lobatto rule, exact and approximate solution of $p(t)$.}
\end{center}
\end{figure}

\end{example}

\newpage

\section{Concluding Remarks}
The paper considers the problem of determining the lowest coefficient that depends on time only,  for a two-dimensional parabolic equation with Ionkin type nonlocal boundary condition and the total energy measurement. The existence and uniqueness of the solution of such an inverse problem and well-posedness of this problem are examined by using the method of series expansion in terms of eigenfunctions and associated functions of corresponding spatial differential operator which is nonself-adjoint and hence the system of eigenfunctions is not complete and must be supplemented by associated functions.    The numerical method can be considered as a suitable combination of finite difference scheme and numerical integration. The traditional approach is the method of uniform finite difference of the equation, initial condition and boundary conditions combined with numerical integration uniformly (trapezoidal or Simpson’s) of the integral overdetermination condition. Consequently, the method of non-uniform finite difference is combined by the Gauss-Lobatto nodes and weights for the integration. The numerical experiments show that the proposed method has a better numerical result than uniform finite difference method. Two numerical examples illustrate how to implement the method. Some tables and figures demonstrate that the method is effective for the 2D heat equation.

\end{document}